\documentclass[a4paper]{amsart}
\usepackage[latin1]{inputenc}
\usepackage{amssymb}
\usepackage{amsmath}
\usepackage{mathrsfs}
\usepackage{eufrak}
\usepackage{amsthm}
\usepackage{amsfonts}
\usepackage{textcomp}
\usepackage{graphicx}
\usepackage[pdftex]{color}
\usepackage{paralist}
\usepackage[shortlabels]{enumitem}
\usepackage{hyperref}
\usepackage{comment}
\usepackage{tikz}
\usepackage{tikz-cd}
\usepackage[arrow, matrix, curve]{xy}

\newtheorem*{corollary*}{Corollary}
\newtheorem{theorem}{Theorem}[section]
\newtheorem*{theorem*}{Theorem}

\newtheorem{corollary}[theorem]{Corollary}
\newtheorem{Question M}[theorem]{Question M}
\newtheorem{Conjecture S}[theorem]{Conjecture S}
\newtheorem{Conjecture K}[theorem]{Conjecture K}
\newtheorem{lemma}[theorem]{Lemma}
\newtheorem{proposition}[theorem]{Proposition}
\newtheorem{question}[theorem]{Question}

\newtheorem*{claim*}{Claim}

\theoremstyle{definition}
\newtheorem{definition}[theorem]{Definition}

\newtheorem*{theorem }{Theorem}

\newtheorem{example}[theorem]{Example}

\theoremstyle{remark}

\numberwithin{equation}{theorem}

\makeatletter
\renewcommand*\env@matrix[1][\
arraystretch]{%
  \edef\arraystretch{#1}%
  \hskip -\arraycolsep
  \let\@ifnextchar\new@ifnextchar
  \array{*\c@MaxMatrixCols c}}
\makeatother

\renewcommand{\mod}{\operatorname{mod}}

\newcommand{\Ext}{\operatorname{Ext}}
\newcommand{\row}{\operatorname{row}}

\newcommand{\End}{\operatorname{End}}
\newcommand{\id}{\operatorname{id}}
\newcommand{\Hom}{\operatorname{Hom}}
\newcommand{\Tr}{\operatorname{Tr}}
\newcommand{\D}{\operatorname{D}}
\newcommand{\grade}{\operatorname{grade}}
\newcommand{\cograde}{\operatorname{cograde}}
\newcommand{\add}{\operatorname{\mathrm{add}}}
\renewcommand{\top}{\operatorname{\mathrm{top}}}
\newcommand{\rad}{\operatorname{\mathrm{rad}}}
\newcommand{\soc}{\operatorname{\mathrm{soc}}}

\newcommand{\red}{\textit{red}}
\newcommand{\wor}{\operatorname{wor}}
\renewcommand{\mod}{\operatorname{mod}}

\newcommand{\idim}{\operatorname{idim}}
\newcommand{\pdim}{\operatorname{pdim}}

\title{Auslander regular algebras and Coxeter matrices}
\date{\today}
\author[V.Kl\'asz]{Vikt\'oria Kl\'asz$^\dagger$}%
\address[V.~Kl\'asz]{Mathematical Institute of the University of Bonn, Endenicher Allee 60, 53115 Bonn, Germany}%
\email{klasz@math.uni-bonn.de}%
\thanks{$^\dagger$Supported by the Deutsche Forschungsgemeinschaft
(DFG, German Research Foundation) under Germany's Excellence Strategy - GZ 2047/1, Projekt-ID
390685813\\
$^\star$Supported by NSERC Discovery Grant RGPIN-2022-03960 and the Canada Research Chairs program, grant number CRC-2021-00120}

\author[R.~Marczinzik]{Ren\'e Marczinzik}%
\address[R.~Marczinzik]{Mathematical Institute of the University of Bonn, Endenicher Allee 60, 53115 Bonn, Germany}
\email{marczire@math.uni-bonn.de}

\author[H. Thomas]{Hugh Thomas$^\star$}
\address[H. Thomas]{LACIM, PK-4211, Universit\'e du Qu\'ebec \`a Montr\'eal
CP 8888, Succ. Centre-ville
Montr\'eal (Qu\'ebec) H3C 3P8 Canada}
\email{thomas.hugh\_r@uqam.ca}

\begin{document}

\begin{abstract}
We show that Iyama's grade bijection for Auslander--Gorenstein algebras coincides with the bijection introduced by Auslander and Reiten. This result uses a new characterisation of Auslander--Gorenstein algebras. Furthermore, we show that the grade bijection of an Auslander regular algebra coincides with the permutation matrix  $P$ in the Bruhat factorisation of the Coxeter matrix. This gives a new, purely linear algebraic interpretation of the grade bijection and allows us to calculate it in a much quicker way than was previously known. 
We give several applications of our main results.
First, we show that the permanent of the Coxeter matrix of an Auslander regular algebra is either 1 or $-1$.
 Second, we obtain a new combinatorial characterisation of distributive lattices among the class of finite lattices. Explicitly, a lattice is distributive if and only if its Coxeter matrix can be written as $PU$ where $P$ is a permutation matrix and $U$ is an upper triangular matrix. 
This also gives a new characterisation of the well-studied rowmotion bijection for distributive lattices.
Other applications include new homological results about modules in blocks of category $\mathcal{O}$ of semisimple  Lie algebras.

\end{abstract}
\maketitle

\section{Introduction}
Auslander--Gorenstein algebras are a non-commutative generalisation of the classical Gorenstein rings from commutative algebra.
Namely, a noetherian ring $A$ is called \emph{Auslander--Gorenstein} if there exists a finite injective coresolution 
$$0 \rightarrow A \rightarrow I^0 \rightarrow I^1 \rightarrow \cdots \rightarrow I^n \rightarrow 0$$
such that the flat dimension of $I^i$ is at most $i$ for all $i=0,1,\dots,n$.
Auslander--Gorenstein algebras of finite global dimension are called \emph{Auslander regular algebras} and they generalise the classical regular rings from commutative algebra, see \cite{B}.
Classical examples of Auslander--Gorenstein rings include Weyl algebras and enveloping algebras of finite-dimensional Lie algebras, see for example \cite{VO}. In this article, we focus on finite-dimensional algebras, where the class of Auslander--Gorenstein algebras includes for example the large class of higher Auslander algebras \cite{I2}. 
Recently it was noted that other important classes of finite-dimensional algebras are also Auslander regular, for example, incidence algebras of distributive lattices \cite{IM} and blocks of category $\mathcal{O}$ \cite{KMM}. 

Auslander--Gorenstein algebras enjoy several nice properties such as extension closedness of their syzygy categories and an explicit characterisation of minimal approximations in such categories, see for example \cite{AR}.
It was first noted by Auslander and Reiten in \cite{AR} that an Auslander--Gorenstein algebra $A$ comes with a natural bijection between indecomposable injective $A$-modules and indecomposable projective $A$-modules. Namely, this bijection sends an indecomposable injective module $I$ to the indecomposable projective module $P=\Omega^{d}(I)$, where $d$ is the projective dimension of $I$. Note that $d$ is always finite since Auslander--Gorenstein algebras are Iwanaga-Gorenstein, that is $\id A_A= \id {}_{A}A< \infty$. It is, however, a highly non-trivial result that $\Omega^{d}(I)$ is always indecomposable, and that the map is well-defined. We call this bijection the \emph{Auslander--Reiten bijection} $\psi:$ $\{ \ $indecomposable injective $A$-modules $\} \rightarrow \{$ indecomposable projective $A$-modules $\}$.
Later it was noted by Iyama in \cite{I} that every Auslander--Gorenstein algebra $A$ has a bijection on the simple $A$-modules, called the \emph{grade bijection}. This grade bijection is given by sending a simple module $S$ to the simple module $\top D \Ext_A^{g_S}(S,A)$ where $g_S:=\grade S=\inf \{i \geq 0 \mid \Ext_A^i(S,A) \neq 0 \}$ is the grade of $S$.
We denote the grade bijection by $\phi: \{$ simple $A$-modules $\} \rightarrow \{$ simple $A$-modules $\}$.

Now let $A$ be a finite-dimensional algebra with a fixed ordering $S(1),\dots,S(n)$ of the simple $A$-modules. In line with this notation, let $P(i)$ and $I(i)$ denote the indecomposable projective and indecomposable injective $A$-modules, respectively.
Then we define the \emph{Auslander--Reiten permutation} $\hat{\psi}: \{1,\dots,n\} \rightarrow \{1,\dots,n\}$ as $\hat{\psi}(i)=j$ when $\psi(I(i))=P(j)$.
We define the \emph{grade permutation} $\hat{\phi}: \{1,\dots,n\} \rightarrow \{1,\dots,n\}$ as $\hat{\phi}(i)=j$ when $\phi(S(i))=S(j)$.
Our first main theorem gives a new characterisation of Auslander--Gorenstein rings and shows that these two permutations coincide:
\begin{theorem} (Theorem \ref{AGcondition} and \ref{gradeBijCharacterisation2})
A finite-dimensional algebra is Auslander--Gorenstein if and only if $\grade S= \pdim I(S)<\infty$ for every simple module $S$ with injective envelope $I(S)$.
Furthermore, the grade permutation coincides with the Auslander--Reiten permutation. 

\end{theorem}
In the following, we assume that a finite-dimensional algebra is a quiver algebra. We denote the primitive idempotents corresponding to the vertices of the quiver by $e_i$.
Note that this is no loss of generality when the field $K$ is algebraically closed, as in this case all finite-dimensional $K$-algebras are Morita equivalent to a quiver algebra and all our notions are invariant under Morita equivalence.
The \emph{Cartan matrix} $\omega_A=(e_{i,j})$ of a finite-dimensional algebra with $n$ simple modules is defined as the $n \times n$-matrix with entries $e_{i,j}=\dim_K e_j A e_i$. 
The \emph{Coxeter matrix} $C_A$ of a finite-dimensional algebra of finite global dimension is defined as $C_A=- \omega_A^T \omega_A^{-1}$. 
Here we use that the Cartan matrix of an algebra with finite global dimension is always invertible over the integers, see for example \cite{Eil}. The study of the Coxeter matrix is its own area in representation theory, called spectral representation theory, see for example the survey article \cite{LP}. An important result is that a path algebra $KQ$ is of finite representation-type if and only if the Coxeter matrix is of finite multiplicative order, which is related to the study of fractionally Calabi-Yau algebras, see for example \cite{CDIM}.
An important property of $C_A$ is that $C_A\cdot \underline{\dim}(P(i))=-\underline{\dim}(I(i))$ for every vertex $i$, see for example \cite[Chapter III.3]{ASS}. 
Note that the calculation of the Cartan and Coxeter matrices depends on the ordering of the simple modules of the algebra.
We call an ordering $S(1),\dots,S(n)$ of the simple $A$-modules \emph{admissible} if the condition $\grade S(i) < \grade S(j) \implies i>j$ is satisfied.
Note that an admissible ordering of the simples always exists: just order them in a non-increasing order according to their grades.

Our next main theorem gives a surprising necessary condition for an algebra to be Auslander regular and a new, purely linear algebraic interpretation of the Auslander--Reiten bijection using the Coxeter matrix. 
For this, recall that a \emph{Bruhat decomposition} of an invertible $n \times n$-matrix $M$ is a factorisation of $M$ as $M=U_1PU_2$, where $P$ is a permutation matrix and $U_1$ and $U_2$ are upper triangular matrices. A Bruhat decomposition always exists and can be obtained via the classical Gau\ss{} algorithm, see for example \cite{OOV}. Moreover, $P$ is uniquely determined by $M$. We call the permutation matrix $P$ which we get if we choose $M=C_A$ the \emph{Coxeter permutation} of the algebra $A$. Note that the Coxeter permutation of an algebra depends in general on the ordering of the simple $A$-modules. We will often identify a permutation matrix with the corresponding permutation in the following.
\begin{theorem}(Theorem \ref{mainresultcoxeter})
Let $A$ be an Auslander regular algebra with admissible ordering of the simple $A$-modules.
Then its Coxeter matrix has a Bruhat decomposition $U_1PU_2$ where $U_1$ is the identity matrix and the Coxeter permutation $P$ corresponds to the grade permutation.
\end{theorem}

As mentioned above, a famous result of Eilenberg states that the Cartan matrix of a finite-dimensional algebra with finite global dimension has determinant  either 1 or $-1$. It is a major open problem, called the Cartan determinant conjecture, whether the Cartan determinant can be $-1$ or not, see \cite{FZ} and \cite{Z}.
We recall that the \emph{permanent} of an $n\times n$-matrix $M=(m_{i,j})_{i,j}$ is defined as
$\sum\limits_{\sigma \in S_n}^{} \prod\limits_{i=1}^{n} m_{i,\sigma(i)}$. So the permanent follows the same definition as the determinant, but we omit the sign. 
It plays an important role in various areas of mathematics such as algebraic combinatorics, complexity theory and statistics. We refer for example to \cite{Mi} for a textbook introduction to permanents.
Our main result gives us the following corollary about the permanent of the Coxeter matrix of an Auslander regular algebra:
\begin{corollary} (Corollary \ref{permanentOfCoxeter})
Let $A$ be an Auslander regular algebra with Coxeter matrix $C_A$.
Then the permanent of $C_A$ is $1$ or $-1$.

\end{corollary}

In fact, this corollary can be used as a quick numerical test to disprove that a given algebra is Auslander regular.

As another major application, we show a new characterisation of distributive lattices among finite lattices using the Coxeter matrix. Here, the Coxeter matrix of a finite poset $R$ is defined as the Coxeter matrix of the incidence algebra of $R$. Note that this definition depends on the ordering of the elements of $R$ but not on the field. For this characterisation, we order the vertices using a linear extension of the order on $R$. This is a more natural order for lattices than the admissible order of the corresponding incidence algebra introduced before.
\begin{theorem} (Corollary \ref{distributiveImpliesL=id} and Theorem \ref{L=idImpliesDistributive})
Let $R$ be a finite lattice with Coxeter matrix $C$, where we order the vertices according to a linear extension of the order on $R.$
Then $R$ is distributive if and only if the first non-zero entry of each row of $C$ is in a different column. In this case, $C$ admits a Bruhat decomposition  $U_1PU_2$ of $C$ in which $U_1$ is the identity matrix, and the Coxeter permutation $P$ coincides with the grade bijection.
\end{theorem}
We remark that the grade bijection for incidence algebras of distributive lattices was shown to coincide with the rowmotion bijection on the vertices of the distributive lattice in \cite{IM}. The rowmotion map is a widely studied object in dynamical algebraic combinatorics, we refer for example to \cite{Str} and \cite{TW} for a survey and more information. Thus, the previous theorem gives us a new interpretation of the rowmotion map for distributive lattices via the Coxeter matrix.

\section{Preliminaries}
We assume that the reader is familiar with the basics on representation theory and homological algebra of finite-dimensional algebras and refer for example to \cite{ASS} and \cite{ARS} for an introduction. For the basics on incidence algebras in representation theory we refer for example to the textbook \cite{S} and \cite{IM}. We assume that all algebras are finite-dimensional algebras over a field $K$ and modules are finitely generated right modules unless otherwise stated. We denote the projective cover of a module $M$ by $P(M)$ and the injective envelope of $M$ by $I(M)$. We denote the natural duality of a finite-dimensional $K$-algebra $A$ by $D=\Hom_K(-,K).$

\subsection{Preliminaries on Auslander--Gorenstein algebras}
\begin{definition}
Let $A$ be a finite-dimensional algebra with minimal injective coresolution
$$0 \rightarrow A \rightarrow I^0 \rightarrow I^1 \rightarrow I^2 \rightarrow\ldots $$
$A$ is called \emph{$n$-Gorenstein} if $\pdim I^i \leq i$ for all $i=0,1,\dots,n-1$. $A$ is called \emph{Auslander--Gorenstein} if $A$ is $n$-Gorenstein for all $n$ and $\idim A< \infty$. An Auslander--Gorenstein algebra $A$ is called \emph{Auslander regular} if it additionally has finite global dimension.
\end{definition}
We remark that an Auslander--Gorenstein algebra is \emph{Iwanaga-Gorenstein}, meaning $\idim A_A = \idim {}_{A}A < \infty$, see \cite[Corollary 5.5]{AR}. The generalised Nakayama conjecture, see \cite{AR2}, states that every simple $A$-module has finite grade. This conjecture is open in general but well-known for Iwanaga-Gorenstein algebras. Since we did not find an explicit reference, we give a quick proof here:
\begin{lemma} \label{finitegradelemma}
Let $A$ be an Iwanaga-Gorenstein algebra. Then every simple $A$-module has finite grade.
\end{lemma}
\begin{proof}
If there is an $i$ with $\Ext_A^i(S,A) \neq 0$, then $S$ has finite grade by definition.
Now assume $S$ is a  simple $A$-module with $\Ext_A^i(S,A)=0$ for all $i \geq 1$. Then $S$ is a maximal Cohen-Macaulay $A$-module (also called Gorenstein projective $A$-module) by \cite[Theorem 2.3.3]{C}. Here we used the assumption that $A$ is Iwanaga-Gorenstein.
Since every maximal Cohen-Macaulay module satisfies $\Hom_A(S,A) \neq 0$ by \cite[Lemma 2.1.4]{C}, the statement of the lemma follows.
\end{proof}

The next definition is due to Iyama.
\begin{definition}
Let $A$ be a finite-dimensional algebra and let 
$$0 \rightarrow A \rightarrow I^0 \rightarrow I^1 \rightarrow I^2 \rightarrow \cdots $$
be a minimal injective coresolution of $A$.
Then $l \geq 0$ is a \emph{dominant number} of $A$ if $\pdim I^i < \pdim I^l$ holds for any $i$ with $0 \leq i < l$.

\end{definition}

The next theorem was shown by Iyama in \cite[Theorem 1.1]{I}:

\begin{theorem} \label{Iyama theorem}
Let $A$ be an $n$-Gorenstein algebra. Then any dominant number $l$ of $A$ with $l<n$ satisfies $\pdim I^l=l$. Moreover, the set of dominant numbers of $A$ smaller than $n$ coincides with that of $A^{op}$ and with the set of grades of $A$-modules which are smaller than $n$.
\end{theorem}

\begin{definition}
An algebra $A$ is called \emph{diagonal Auslander--Gorenstein} (\emph{diagonal Auslander regular}) if it is Auslander--Gorenstein (Auslander regular) and $\pdim X=l$ for every direct summand $X$ of $I^l$.
$A$ simple module is called \emph{perfect} if $\grade S=\pdim S$.

\end{definition}
\subsection{The Bruhat decomposition of an invertible matrix}
We recall the following basic fact from linear algebra, whose proof relies mainly on Gaussian elimination. 
\begin{theorem}
Let $M$ be an invertible real $n \times n$ matrix. Then there is a factorisation $M=U_1PU_2$, where $P$ is a permutation matrix and $U_1$ and $U_2$ are upper triangular matrices. The permutation matrix $P$ is uniquely determined by $M$.
\end{theorem}
This decomposition is called the \emph{Bruhat decomposition} of $M$.
We remark that while the permutation matrix $P$ is uniquely determined by $M$, this is not true in general for $U_1$ and $U_2$. But by assuming some additional conditions on $U_1$ and $U_2$ one can obtain a unique form of the Bruhat decomposition, we refer to \cite{OOV} for details. For the relevance of the Bruhat decomposition and generalisations, we refer for example to \cite[Chapter 27]{Bu}.
\subsection{A survey of bijections}
Let $A$ be a finite-dimensional algebra with a fixed ordering $S(1),\dots,S(n)$ of the simple $A$-modules. In this subsection, we survey some bijections that exist for some special classes of finite-dimensional algebras and related combinatorial objects, such as distributive lattices and indecomposable summands of cluster-tilting modules. Only the first bijection, namely, the Coxeter permutation, is new.

\subsubsection{The Coxeter permutation}
Let $A$ be an algebra of finite global dimension with simple $A$-modules $S(1), \ldots, S(n)$ and Coxeter matrix $C$.
Let $C=U_1PU_2$ be a Bruhat decomposition of $C$. Then we call the permutation $\pi:\{1,\dots,n\} \rightarrow \{1,\dots,n\}$ corresponding to $P$ the \emph{Coxeter permutation} of $A$.
We note that the Coxeter permutation depends on the ordering of the simple modules of the algebra $A$. 
When $P$ is a finite poset with a fixed linear extension, then the Coxeter permutation of the incidence algebra of $KP$ is called the \emph{Coxeter permutation} of $P$. This definition depends in general on the linear extension, but we will see that it is independent of the choice of the linear extension when $P$ is a distributive lattice.
Here we obtain a permutation $\pi$ on $\{1,\dots,n\}$ from an $n \times n$-permutation matrix by sending $i$ to the index of the unique row which has a non-zero entry in the $i$-th column, see Example \ref{examplerowmotioncoxeter}. In \cite{DJMSSTW} the Coxeter permutation is studied under the name echelonmotion in a more combinatorial context for incidence algebras of posets. In the forthcoming article \cite{KKM} it will be shown that the Coxeter permutation coincides with Ringel's homological bijection \cite{R} for linear Nakayama algebras with the canonical ordering of the simple modules.

\subsubsection{The grade bijection of Iyama}
\label{gradeBijection} If $A$ is an Auslander--Gorenstein algebra, the grade bijection is defined as $$\phi: \{\text{ simple A-modules }\}/_{\cong} \rightarrow \{ \text{ simple A-modules } \}/_{\cong}$$  $$S\mapsto \top(D\Ext_A^{g_S}(S,A)),$$ where $g_S:=\inf \{i \geq 0 \mid \Ext_A^i(S,A) \neq 0 \}$ is the grade of $S$. The cograde of the simple module $S$ is defined dually, i.e. as $\inf \{i \geq 0 \mid \Ext_A^i(D(A),S) \neq 0 \}.$ An important property of $\phi$ is that cograde of $\phi(S)$ equals the grade of $S$. See  \cite[Theorem 2.10]{I}.

We define the \emph{grade permutation} $\hat{\phi}: \{1,\dots,n\} \rightarrow \{1,\dots,n\}$ as $\hat{\phi}(i)=j$ when $\phi(S(i))=S(j)$.

\subsubsection{The Auslander--Reiten bijection for Auslander--Gorenstein algebras}
\label{ARbijection}
Let $A$ be an Auslander--Gorenstein algebra. Then there is a bijection 
$$\psi: \ \{ \ \text{indecomposable injective A-modules }\}/_{\cong} \rightarrow \{\text{ indecomposable projective A-modules }\}/_{\cong}$$ $$I\mapsto P=\Omega^{d}(I),$$ where $d$ is the projective dimension of $I$. Moreover, $P$ has then injective dimension equal to $d$. 

The inverse of this bijection $\psi^{-1}$ maps an indecomposable projective $A$-module $P$ to $I=\Omega^{-d}(P)$ where $d=\idim(P).$ See \cite[Proposition 5.4]{AR}. 

We define the \emph{Auslander--Reiten permutation} $\hat{\psi}: \{1,\dots,n\} \rightarrow \{1,\dots,n\}$ as $\hat{\psi}(i)=j$ when $\psi(I(i))=P(j)$ for indecomposable injective modules $I(i)$ and indecomposable projective modules $P(j)$.

\subsubsection{Higher Auslander--Reiten translates}
Let $B$ be an $n$-representation-finite algebra with $n$-cluster tilting object $M$. Then there is a map on the indecomposable objects $X$ in $\add(M)$ sending $X$ to $\tau_n(X)$ if $X$ is not projective and to $\nu(X)$ if $X$ is projective. Thus this map acts as the higher Auslander--Reiten translate $\tau_n=\tau \Omega^{n-1}$  on the indecomposable modules in the cluster tilting subcategory $\add(M)$ except for the projective modules, where it acts as $\nu$, the Nakayama functor sending a projective module $P$ to $D \Hom_B(P,B)$. We refer for example to \cite{I2} for more details.
This map on the indecomposable modules in $\add(M)$ corresponds to the grade bijection in the higher Auslander algebra $\End_B(M)$, see \cite[Theorem 4.6]{MTY}. 

\subsubsection{The Nakayama permutation}
Let $A$ be a Frobenius algebra. The Nakayama permutation sends $i$ to $\pi(i)$, where $\soc(P(i))=\top(P({\pi(i)}))$ for an indecomposable projective module $P(i)$, see for example \cite[Chapter 4]{SkoYam}. It turns out that the Nakayama permutation is a special case of the Auslander--Reiten permutation since the injective resolution of the indecomposable projective $P(i)$ is given by $0 \rightarrow P(i) \rightarrow I(\pi(i)) \rightarrow 0$.
For example for preprojective algebras of Dynkin type, this bijection is given by negative one times the action of the longest element of the Coxeter group on the simple roots, see for example \cite{G} for more details.

\subsubsection{The rowmotion bijection}
 The \emph{rowmotion bijection} for a distributive lattice $L$, given as the set of order ideals of a poset $P$, is defined as $\row(x)$ being the order ideal generated by the minimal elements in $P\setminus x$ for an order ideal $x$ of $P$. The rowmotion bijection has appeared under many different names and is a useful tool in the study of combinatorial properties of posets arising for example in Lie theory, see for example the references \cite{Str} and \cite{TW} that also contain a historical background on the rowmotion bijection.
It turns out that the rowmotion bijection is a special case of the grade permutation on the incidence algebra $A=KL$ of the distributive lattice $L$. 
Explicitly, if $L$ is a distributive lattice then $A$ is Auslander--Gorenstein. Vertices of $L$ label the simple modules of $A$, and $row(x)=\hat{\phi}(x)$ for every vertex $x$ of L. See \cite{IM}.

\begin{example} \label{examplerowmotioncoxeter}
We give one example that illustrates the Coxeter permutation and the rowmotion bijection.
Let $R$ be the poset 
\[\begin{tikzcd}
	1 && 2 \\
	& 0 && .
	\arrow[from=2-2, to=1-1]
	\arrow[from=2-2, to=1-3]
\end{tikzcd}\]
Then the distributive lattice $L$ of order ideals of $R$ is given as follows:
\[\begin{tikzcd}
	& {\{0,1,2\}} \\
	{\{0,1\}} && {\{0,2\}} \\
	& {\{0\}} \\
	& \emptyset
	\arrow[from=2-1, to=1-2]
	\arrow[from=2-3, to=1-2]
	\arrow[from=3-2, to=2-1]
	\arrow[from=3-2, to=2-3]
	\arrow[from=4-2, to=3-2]
\end{tikzcd}\]
The rowmotion bijection sends $\emptyset$ to $\{0\}$, $\{0\}$ to $\{0,1,2\}$, $\{0,1\}$ to $\{0,2\}$, $\{0,2\}$ to $\{0,1\}$ and $\{0,1,2\}$ to $\emptyset$.
We now look at the incidence algebra $A=KL$ of the distributive lattice $L$ with the following labeling of the vertices:
\[\begin{tikzcd}
	& 5 \\
	3 && 4 \\
	& 2 \\
	& 1
	\arrow[from=2-1, to=1-2]
	\arrow[from=2-3, to=1-2]
	\arrow[from=3-2, to=2-1]
	\arrow[from=3-2, to=2-3]
	\arrow[from=4-2, to=3-2]
\end{tikzcd}\]
Then the Cartan matrix of $A$ is given by
\[ \left( \begin{array}{ccccc}
1 & 0 & 0 & 0 & 0 \\
1 & 1 & 0 & 0 & 0 \\
1 & 1 & 1 & 0 & 0 \\
1 & 1 & 0 & 1 & 0 \\
1 & 1 & 1 & 1 & 1 \end{array} \right) \]
and the Coxeter matrix of $A$ is given by 
\[ \left( \begin{array}{ccccc}
0 & 0 & 0 & 0 & -1 \\
1 & 0 & 0 & 0 & -1 \\
0 & 0 & 0 & 1 & -1 \\
0 & 0 & 1 & 0 & -1 \\
0 & -1 & 1 & 1 & -1\end{array}  \right) = \left( \begin{array}{ccccc}
0 & 0 & 0 & 0 & 1 \\
1 & 0 & 0 & 0 & 0 \\
0 & 0 & 0 & 1 & 0 \\
0 & 0 & 1 & 0 & 0 \\
0 & 1 & 0 & 0 & 0\end{array}  \right)  \left( \begin{array}{ccccc}
1 & 0 & 0 & 0 & -1 \\
0 & -1 & 1 & 1 & -1 \\
0 & 0 & 1 & 0 & -1 \\
0 & 0 & 0 & 1 & -1 \\
0 & 0 & 0 & 0 & -1\end{array}  \right) .\]
Here we also spelled out the Bruhat decomposition of the Coxeter matrix. From this, we see that the rowmotion bijection on $L$ coincides with the Coxeter permutation of the incidence algebra $A$. This fact holds in general for distributive lattices as we will see later.
\end{example}

\section{Grade bijection}
We give a new characterisation of Auslander--Gorenstein algebras. First, we need two lemmas.
	\begin{lemma}
		\label{extformula}
		Let~$A$ be a finite-dimensional algebra with a simple $A$-module $S$.
		\begin{enumerate}  
			\item  Let $M$ be an $A$-module with minimal projective resolution
			$$\cdots \rightarrow P_i \rightarrow \cdots \rightarrow P_1 \rightarrow P_0 \rightarrow M \rightarrow 0.$$
			For $l \geq 0$, $\Ext_A^l(M,S) \neq 0$ if and only if there is a surjection $P_l \rightarrow S$.
			\item   Dually, let
			$$0 \rightarrow M \rightarrow I^0 \rightarrow I^1 \rightarrow \cdots \rightarrow I^i \rightarrow \cdots $$
			be a minimal injective coresolution of~$M$.
			For $l \geq 0$, $\Ext_A^l(S,M) \neq 0$ if and only if there is an injection $S \rightarrow I_l$.
		\end{enumerate}
	\end{lemma}
	
	\begin{proof}
		See for example~\cite[Corollary~2.5.4]{Ben}.
	\end{proof}
\begin{lemma}\label{gradeSimpleSubmod} For every $M\in \mod A$ and every simple submodule $S$ of $M$ with $\grade(S)<\infty$ we have
$$\grade(S)\le \pdim(M).$$
\end{lemma}
\begin{proof}
If $\grade(S)=0$ then the inequality clearly holds, so let us assume that $\infty>g:=\grade(S)>0.$ By the proposition in Appendix D.1 in \cite{R}, there exists an $N\in \mod A$ such that $\idim(N)=g$ and $\Ext^g_A(S,N)\neq 0.$ More precisely, this is true for $N=\tau_g(S)=\D\Tr\Omega^{g-1}(S).$ As the proof of this statement is relatively short, we will include it here.

Let us consider a minimal projective resolution of $S$ truncated at the $g$-th term: 

$$P_g\xrightarrow{f_g}P_{g-1}\rightarrow\ldots\rightarrow P_1\xrightarrow{f_1} P_0 \rightarrow 0.$$
By the assumption $\grade(S)=g,$ we have that $\Ext^{i}(S,A)=0$ for every $0\le i<g$. Thus, we can apply $\Hom_A(-,A)$ to this truncated resolution and get an exact sequence of $A^{op}$-modules
$$P_g^*\xleftarrow{f_g^*}P_{g-1}^*\leftarrow\ldots\leftarrow P_1^*\xleftarrow{f_1^*} P_0^* \leftarrow 0.$$
By definition, the cokernel of $f_g$ is $\Omega^{g-1}(S)$ and so the cokernel of $f_g^*$ is $\Tr(\Omega^{g-1}(S)).$ So 
$$0\leftarrow \Tr\Omega^{g-1}(S)\leftarrow P_g^*\xleftarrow{f_g^*}P_{g-1}^*\leftarrow\ldots\leftarrow P_1^*\xleftarrow{f_1^*} P_0^* \leftarrow 0$$
is a minimal projective resolution of the $A^{op}$-module $\Tr(\Omega^{g-1}(S)).$ By applying the duality $D$ to this sequence, we get a minimal injective coresolution of $N=\tau_g(S).$ Moreover, $\D P_0^*\cong I(S)$ is an injective envelope of $S,$ since $P_0$ was a projective cover of $S$. This proves $\idim(N)=g$ and $\Ext^{g}(S,N)\neq0.$ For the last statement we used Lemma \ref{extformula} (2). 

Now let us turn our attention to the proof of our lemma and consider the following short exact sequence of $A$-modules
$$0\to S\to M\to M/S \to 0,$$
where the first map is the canonical inclusion of the submodule $S$ into $M.$ Applying $\Hom_A(-,N)$ to this, we get the usual long exact $\Ext$ sequence. This, combined with $\Ext^g_A(S,N)\neq 0$ implies that $\Ext^{g}_A(M,N)\neq 0$ or $\Ext^{g+1}_A(M/S,N)\neq 0$ has to hold. However, since $\idim(N)=g$ we have $\Ext^{g+1}_A(M/S,N)=0,$ so  $\Ext^{g}_A(M,N)\neq 0,$ and thus $\pdim(M)\ge g.$
\end{proof}

\begin{theorem}\label{AGcondition}
Let $A$ be a finite-dimensional algebra. Then the following are equivalent:
\begin{enumerate}[(i)]
    \item $A$ is Auslander--Gorenstein.
    \item For every simple $A$-module $S$ the following holds: $\grade(S)=\pdim(I(S))<\infty.$
\end{enumerate}
\end{theorem}
\begin{proof}
Let us denote a minimal injective coresolution of $A$ by 
$$0\rightarrow A \rightarrow I^0 \rightarrow I^1 \rightarrow \ldots.$$
    Assume that $A$ is Auslander--Gorenstein. Then $A$ is Iwanaga-Gorenstein and  we know that $\grade(S)<\infty$ by Lemma \ref{finitegradelemma}. Now we can apply Lemma \ref{gradeSimpleSubmod} to obtain $\grade(S)\le \pdim(I(S)).$

    We know that the grade of a simple module $S$ is the smallest integer $r$ such that $I(S)$ is a direct summand of $I^r.$  Hence, $I(S)$ appears as a direct summand of $I^{\grade(S)}.$ Because $A$ is assumed to be $n$-Gorenstein for every $n,$ we have $\pdim(I(S))\le\pdim(I^{\grade(S)})\le \grade(S).$ 

    For the other direction, let us assume that $(ii)$ holds for $A.$ Let $S$ be an arbitrary simple $A$-module. By assumption, $\grade(S)<\infty.$ Then by the above discussion we know that $I(S)$ can only appear as a direct summand of $I^r$ if $r\ge \grade(S).$ This, combined with $\grade(S)=\pdim(I(S))$ for every simple $S,$ yields that $\pdim(I^r)\le r$ holds for every natural number $r.$ The assumption that every injective module has finite projective dimension implies that every projective $A^{op}$-module has finite injective dimension. We get from \cite[Corollary 5.5.(b)]{AR} that $\idim A_A = \idim {}_{A}A$ for algebras $A$ that are $n$-Gorenstein for all $n$ and these observations already imply that $(i)$ holds. 
\end{proof}
We give two corollaries:

\begin{corollary} \label{corollarydiagonal}
An Auslander regular algebra is diagonal if and only if every simple module is perfect. 

\end{corollary}
\begin{proof}
First, assume that $A$ is diagonal Auslander regular.
Then in a minimal injective coresolution of $A$:
$$0 \rightarrow A \rightarrow I^0 \rightarrow I^1 \rightarrow \cdots \rightarrow I^n \rightarrow 0$$
every direct summand $X$ of a term $I^i$ has projective dimension equal to $i$.
Now, let $S$ be a simple module with injective envelope $I(S)$ and assume that $I(S)$ has projective dimension $l$. Since every simple module of this algebra has finite grade by Lemma \ref{finitegradelemma}, $I(S)$ must show up somewhere in the minimal injective resolution of $A.$ Thus, $I(S)$ must be a direct summand of $I^l$ and cannot appear as any other direct summand of a term $I^j$ for $j \neq l$. Applying Lemma \ref{extformula}, we get that $\Ext_A^l(S,A) \neq 0$, while $\Ext_A^j(S,A)=0$ for $j \neq l$. As $A$ is Auslander regular and so $\pdim(S)<\infty$, we can use the formula $\pdim(S)=\max\{ \ j \ | \ \Ext_A^j(S,A)\neq 0\}$ from \cite[Lemma VI.5.5]{ARS} to obtain that $S$ is perfect.

Now assume that every simple $A$-module $S$ is perfect.
When $I(S)$ denotes the injective envelope of $S$ with projective dimension $l$, then 
\begin{equation} \label{equation1}
    \pdim S= \grade S= \pdim I(S)=l.
\end{equation}
Here we used Theorem \ref{AGcondition}.
Thus, $\Ext_A^l(S,A) \neq 0$ and $\Ext_A^j(S,A)=0$ for $j\neq l$.
By Lemma \ref{extformula}, this means that $I(S)$ only appears as a direct summand of $I^l$, and furthermore, we know by (\ref{equation1}) that $\pdim I(S)=l$, showing that $A$ is diagonal.
\end{proof}

We remark that in \cite[3.6.2]{I3}, Iyama posed the following question:
\begin{question}
Let $A$ be Auslander regular and diagonal. Is then $A^{op}$ diagonal as well?
\end{question}
Using Corollary \ref{corollarydiagonal}, we can reformulate Iyama's question:
\begin{question}
Let $A$ be an Auslander regular algebra such that every simple right $A$-module is perfect. Is then every simple left $A$-module perfect as well?
\end{question}
\begin{corollary}
The set dominant numbers of an Auslander--Gorenstein algebra coincides with the set of projective dimensions of indecomposable injective modules.

\end{corollary}
\begin{proof}
This follows directly from Theorem \ref{Iyama theorem}, which tells us that the dominant numbers all come from the set of projective dimensions of injective modules in an Auslander--Gorenstein algebra. Thus, we only need to show, that all the projective dimensions of indecomposable injective modules are dominant numbers.

By Theorem \ref{AGcondition}, the projective dimension of an indecomposable injective module coincides with the grade of its socle. This finishes the proof, as Theorem \ref{Iyama theorem} tells us that the grade of any $A$-module is a dominant number. 
\end{proof}
We give an application for blocks of category $\mathcal{O}$. Note that these are most often wild, and thus, their modules cannot be classified in general.
For background information on blocks of category $\mathcal{O}$ we refer for example to \cite{Hu}, and for more on Lusztig's $a$-function see for example \cite{Maz}.

\begin{proposition}
Let $A$ be a block of category $\mathcal{O}$.
The set of grades of the $A$-modules coincides with two times the values of Lusztig's $a$-function. 
\end{proposition}
\begin{proof}
By Theorem \ref{Iyama theorem}, the set of grades of $A$-modules coincides with the set of grades of simple $A$-modules since $A$ is Auslander regular by the main result of \cite{KMM}.
Now, by our Theorem \ref{AGcondition}, the grades of simple $A$-modules coincide with the projective dimensions of the indecomposable injective $A$-modules.
The set of values of those projective dimensions coincides with two times the values of Lusztig's $a$-function by \cite[Theorem 20]{Maz}.
\end{proof}

Our next goal is to show that the grade permutation coincides with the Auslander--Reiten permutation. For this, we first need some preliminary work.
\begin{proposition} \label{gradeBijCharacterisation}

Let $A$ be Auslander--Gorenstein. Let $S$ be a simple $A$-module, with injective envelope $I$ and grade $r$. Let $\mathcal{P}$ be the set of indecomposable projective $A$-modules such that $I$ appears in degree $r$ in the injective resolution of $P$. Among the projectives of $\mathcal{P}$, there is a unique one whose top has cograde $r$. This is $\phi(S)$. The injective $I$ appears exactly once in the coresolution of the projective cover of
$\phi(S)$.
\end{proposition}
\begin{proof}
Consider $\Ext_A^r(S,A)$. To understand its $A^{op}$-module structure, we can think of it as a representation of the Gabriel quiver of $A^{op}$. For $e$ a primitive idempotent of $A^{op}$, we have that $e \Ext_A^r(S,A)=\Ext_A^r(S,eA)$. That is to say, the vector space associated to the vertex corresponding to $e$ is given by $\Ext_A^r(S,eA)$. Thus, the vertices over which $\Ext_A^r(S,A)$ 
has support are exactly the elements of $\mathcal{P}$.
By part (ii) of the proof of Theorem 2.10 in \cite{I}, all the simples appearing in a composition series of $\Ext_A^r(S,A)$ as an $A^{op}$-module other than its socle $(D \phi(S))$ are of grade greater than $r$. Dually, all but one of the tops of the modules from $\mathcal{P}$ is of cograde greater than $r$, while one is of cograde $r$, and it is this one which is $\phi(S)$.
The same fact from the proof of Theorem 2.10 in \cite{I} implies that the simple $D \phi(S)$ can only appear once in a composition series for $\Ext_A^r(S,A)$. If we write $P$ for the projective cover of $\phi(S)$, this implies that $\Ext_A^r(S,P)$ is one-dimensional, and the final statement of the proposition follows since by the dual statement of Theorem \ref{AGcondition} we know that $\idim(P)=r$ so $I$ doesn't appear in degrees higher than $r$.
\end{proof}

\begin{lemma}

Let $A$ be Auslander--Gorenstein. Let $I$ be an injective module of projective dimension $r$. Among the indecomposable projective modules
such that $I$ appears in their minimal coresolution, there is a unique one which is of injective dimension $r$, while all the others, if they exist, are of injective dimension strictly greater than $r$.
\end{lemma}
\begin{proof}
Let $S$ be the socle of $I$. By Theorem \ref{AGcondition}, we have $\grade(S)=\pdim(I)=r.$ This means that $I$ appears first in degree $r$ of the injective coresolution of $A.$ 

If for a projective $A$-module $P$, $I$ appears in degree $d>r$ in its minimal injective coresolution, then $\idim(P)\ge d>r.$ So it is enough to concentrate on the indecomposable projectives for which $I$ appears in their minimal coresolution in degree $r$. These are exactly the elements of the set $\mathcal{P}$ from Proposition \ref{gradeBijCharacterisation}. From this proposition, we know that there is a unique $P\in \mathcal{P}$ whose
top has cograde $r$, while the others are higher. This means that $P$ appears in
degree $r$ of the projective resolution of $D({}_AA)$, while the other elements of $\mathcal{P}$
appear in higher degrees. The injective dimension of $P$ is therefore at most $r,$ as $A^{op}$ is also
Auslander--Gorenstein if $A$ is. As $I$ appears in degree $r$ of the injective coresolution of $P,$ $\idim(P)=r.$

Now let $P'\in\mathcal{P}\backslash\{P\}.$ Then by the grade bijection, there exists an $S'$ such that $\phi(S')=\top(P')$ and $\grade(S')=\cograde(\top(P'))=r'>r.$ By the previous argument applied to $I(S')$ and using Theorem \ref{AGcondition} again, we have that $\idim(P')=\pdim(I(S'))=\grade(S')=r'>r.$
\end{proof}

\begin{corollary}\label{gradeBijCorollary} 
Let $A$ be Auslander--Gorenstein. Let $I$ be an indecomposable
injective of projective dimension $r$. Then the grade bijection sends the socle of $I$ to
the top of the unique projective $P$ of injective dimension $r$ such that $I$ appears in
the minimal injective coresolution of $P$.
\end{corollary}

The next theorem tells us that the grade permutation coincides with the Auslander--Reiten permutation for Auslander--Gorenstein algebras.
\begin{theorem}\label{gradeBijCharacterisation2}
Let $A$ be Auslander--Gorenstein. Let $I$ be an injective indecomposable of projective dimension $r$. Let $P$ be the projective indecomposable with $\top(P)=\phi(\soc(I))$. Then $P$ is of injective dimension $r$. The $r$-th (final) term in the projective resolution of $I$ is $P$, and the $r$-th (final) term in the injective resolution of $P$ is $I$.
\end{theorem}

\begin{proof} Since $I$ has projective dimension $r$ and $\pdim(I(S))=\grade(S)$ for every simple module $S$ in an Auslander--Gorenstein algebra by Theorem \ref{AGcondition}, Proposition \ref{gradeBijCharacterisation} tells us that $I$ appears in the $r$-th term
of the injective coresolution of $P$ and by Corollary \ref{gradeBijCorollary}, $\idim(P)=r.$ 
By \cite[Prop. 5.4.(b)]{AR}, we know that since $A$ is $r$-Gorenstein, the functor $\Omega^{-r}$ gives a bijection between (the isomorphism classes of) the indecomposable projective $A$-modules of injective dimension $r$ and (the isomorphism classes of) the indecomposable injective $A$-modules of projective dimension $r.$ Moreover, $\Omega^r$ gives the inverse bijection between these two sets. Thus, $\Omega^{-r}(P)$ is indecomposable, and so $\Omega^{-r}(P)\cong I.$ The inverse bijection is given by $\Omega^r$, so $\Omega^{r}(I)\cong P.$
\end{proof}

\begin{theorem} \label{mainresultcoxeter}
Let $A$ be an Auslander regular algebra with an admissible ordering of the simple modules and with Coxeter matrix $C$. Then the Bruhat decomposition $U_1PU_2$ of $C$ is given by $C=PU_2$, where $P$ corresponds to the grade bijection of $C$ and $U_2$ is upper triangular with diagonal entries $1$ or $-1$.
\end{theorem}
\begin{proof}
Let
us number the projective, simple, and injective modules of A, so that if the grade
of $S(i)$ is less than the grade of $S(j)$ then $i > j$. In other words, we fix an admissible ordering.  We use the notation $P(i)$ and $I(i)$ for the indecomposable projective and indecomposable injective $A$-modules, respectively.

Let $M$ be the matrix whose $ij$ entry counts the occurrences of $I(j)$ in the injective coresolution of $P(i)$ with signs: i.e., each occurrence of $I(j)$ in even degree contributes positively, and each occurrence of $I(j)$ in odd degree contributes negatively. Write $\hat I$ for the matrix whose rows are the dimension vectors of the injective modules, and $\hat P$ for the matrix whose rows are the dimension vectors of the projective modules. Because the dimension vector of $P(i)$ equals the alternating sum of the dimension vectors of the injective modules appearing in the coresolution of $P(i)$, we have that $e_i \cdot M \hat I=\underline{\dim} P(i)$, so $M\hat I=\hat P$. Now, using the fact that $\hat P^T=\hat I$ and vice versa, we find that $M \cdot \underline{\dim}P(i)=\underline{\dim}I(i)$. We therefore have that $M=-C$, since by definition the Coxeter matrix $C$ is the matrix that satisfies $C\cdot \underline{\dim} P(i)=-\underline{\dim}I(i)$.

Since $A$ is Auslander--Gorenstein and by Theorem  \ref{gradeBijCharacterisation2}, the entry $M_{ij}$ is nonzero only if either the projective dimension of $I(j)$ is less
than the injective dimension of $P(i)$, or $\phi(S(j) ) = S(i).$ Consider now the matrix $\phi^{-1} M$.
Its $ij$ entry is nonzero if and only if $M_{\phi(i),j} \neq 0.$ 
This occurs only if $\pdim I(j) < \idim P(\phi(i))$ or $i=j.$ The first inequality is equivalent to $\grade S(j) < \cograde S(\phi(i))=\grade S(i)$. Here we used Theorem \ref{AGcondition} and the property of $\phi$ discussed in \ref{gradeBijection}. By our ordering of the simples, this tells us that $\phi^{-1}M$ is upper triangular. The diagonal entries are either 1 or $-1$ by Theorem \ref{gradeBijCharacterisation2}.
\end{proof}

The next corollary provides a purely linear algebraic way to show that a given algebra is not Auslander regular. We emphasise that for the next corollary, the ordering of the simples does not matter.
\begin{corollary}\label{permanentOfCoxeter}
Let $A$ be an Auslander regular algebra with Coxeter matrix $C$. 
Then the permanent of $C$ is $1$ or $-1$.

\end{corollary}
\begin{proof}
The permanent of $C$ does not depend on the ordering of the simples, as it is invariant under multiplication with permutation matrices. So we can assume that the simples are in an admissible order, and then by Theorem \ref{mainresultcoxeter}, $C=\Pi U$ with $\Pi$ a permutation matrix and $U$ an upper triangular matrix with diagonal entries 1 or $-1$.
 Using again the fact that the permanent is invariant under multiplication with permutation matrices, we obtain that the permanent of $C$ is equal to the permanent of $U$, which is just the product of the diagonal entries.

\end{proof}

\begin{example}
The following poset is 2-Gorenstein, but it is not Auslander regular since the permanent of the Coxeter matrix is equal to $-1501$.
\begin{tikzpicture}[>=latex,line join=bevel,]
\begin{tikzpicture}[>=latex,line join=bevel,]
\node (node_0) at (51.0bp,6.5bp) [draw,draw=none] {$0$};
  \node (node_1) at (66.0bp,55.5bp) [draw,draw=none] {$1$};
  \node (node_2) at (6.0bp,55.5bp) [draw,draw=none] {$2$};
  \node (node_3) at (36.0bp,55.5bp) [draw,draw=none] {$3$};
  \node (node_5) at (96.0bp,55.5bp) [draw,draw=none] {$5$};
  \node (node_4) at (6.0bp,104.5bp) [draw,draw=none] {$4$};
  \node (node_7) at (96.0bp,104.5bp) [draw,draw=none] {$7$};
  \node (node_8) at (66.0bp,104.5bp) [draw,draw=none] {$8$};
  \node (node_6) at (36.0bp,104.5bp) [draw,draw=none] {$6$};
  \node (node_9) at (51.0bp,153.5bp) [draw,draw=none] {$9$};
  \draw [black,->] (node_0) ..controls (54.908bp,19.746bp) and (58.351bp,30.534bp)  .. (node_1);
  \draw [black,->] (node_0) ..controls (38.73bp,20.316bp) and (27.098bp,32.464bp)  .. (node_2);
  \draw [black,->] (node_0) ..controls (47.092bp,19.746bp) and (43.649bp,30.534bp)  .. (node_3);
  \draw [black,->] (node_0) ..controls (63.27bp,20.316bp) and (74.902bp,32.464bp)  .. (node_5);
  \draw [black,->] (node_1) ..controls (50.238bp,68.847bp) and (32.656bp,82.619bp)  .. (node_4);
  \draw [black,->] (node_1) ..controls (73.953bp,68.96bp) and (81.16bp,80.25bp)  .. (node_7);
  \draw [black,->] (node_1) ..controls (66.0bp,68.603bp) and (66.0bp,79.062bp)  .. (node_8);
  \draw [black,->] (node_2) ..controls (6.0bp,68.603bp) and (6.0bp,79.062bp)  .. (node_4);
  \draw [black,->] (node_2) ..controls (13.953bp,68.96bp) and (21.16bp,80.25bp)  .. (node_6);
  \draw [black,->] (node_2) ..controls (21.762bp,68.847bp) and (39.344bp,82.619bp)  .. (node_8);
  \draw [black,->] (node_3) ..controls (28.047bp,68.96bp) and (20.84bp,80.25bp)  .. (node_4);
  \draw [black,->] (node_3) ..controls (36.0bp,68.603bp) and (36.0bp,79.062bp)  .. (node_6);
  \draw [black,->] (node_3) ..controls (51.762bp,68.847bp) and (69.344bp,82.619bp)  .. (node_7);
  \draw [black,->] (node_4) ..controls (18.27bp,118.32bp) and (29.902bp,130.46bp)  .. (node_9);
  \draw [black,->] (node_5) ..controls (80.238bp,68.847bp) and (62.656bp,82.619bp)  .. (node_6);
  \draw [black,->] (node_5) ..controls (96.0bp,68.603bp) and (96.0bp,79.062bp)  .. (node_7);
  \draw [black,->] (node_5) ..controls (88.047bp,68.96bp) and (80.84bp,80.25bp)  .. (node_8);
  \draw [black,->] (node_6) ..controls (39.908bp,117.75bp) and (43.351bp,128.53bp)  .. (node_9);
  \draw [black,->] (node_7) ..controls (83.73bp,118.32bp) and (72.098bp,130.46bp)  .. (node_9);
  \draw [black,->] (node_8) ..controls (62.092bp,117.75bp) and (58.649bp,128.53bp)  .. (node_9);
\end{tikzpicture}
\end{tikzpicture}

\end{example}

We give an application of our main result, Theorem \ref{mainresultcoxeter}, and calculate the grade bijection for blocks of category $\mathcal{O}$.
Recall that a \emph{simple-preserving duality} is a contravariant anti-equivalence $H$ of $\mod A$ which preserves the isomorphism classes of simple $A$-modules.
For example all blocks of category $\mathcal{O}$ have a simple-preserving duality, see for example \cite{BGS,Hu}.
\begin{proposition}
Let $A$ be an Auslander regular algebra with a simple-preserving duality. Then the grade bijection of $A$ is the identity.
\end{proposition}
\begin{proof}
We first show that the Cartan matrix of $A$ is a symmetric matrix.
Let $H$ denote the simple-preserving duality. 

Applying $H$ to the projective cover $P(i) \rightarrow S(i) \rightarrow 0$ and using $H(S(i))=S(i)$ gives
$0 \rightarrow S(i) \rightarrow H(P(i))$.
But $H(P(i))$ is the injective envelope of $S(i)$ and thus $H(P(i))=I(i)$.
Now as $K$-vector spaces we have:
$$e_j A e_i= \Hom_A(e_i A, e_j A)= \Hom_A(P(i),P(j)) \cong \Hom_A(H(P(j)),H(P(i))) \cong \Hom_A(I(j),I(i))= $$
$$\Hom_A(D(Ae_j),D(Ae_i)) \cong \Hom_A(Ae_i,Ae_j) \cong e_i A e_j.$$
Thus, the Cartan matrix of $A$ is symmetric.
By definition of the Coxeter matrix, we get $C_A=- \id$, which implies that the permutation matrix $P$ is the identity matrix in the Bruhat decomposition of $C_A$. Thus, by Theorem \ref{mainresultcoxeter}, the grade bijection is the identity.
\end{proof}

\section{Characterisation of distributivity for lattices using the Coxeter matrix}

The main result of this section shows that a finite lattice $R$ is distributive if and only if its Coxeter matrix (whose rows and columns are ordered with respect to a linear extension of $R$) has a Bruhat decomposition $U_1PU_2$ where $U_1$ is the identity matrix.

\begin{lemma} \label{reduction} Let $R$ be a lattice whose Bruhat decomposition $U_1PU_2$ has $U_1=\id$. Let $I$ be an upper interval of $R$. Then the same property holds for $I$.\end{lemma}

\begin{proof} Let $I=[z,\hat 1]$. 
Let $C$ be the Coxeter matrix of $R$ with respect to some fixed order.
Let $x$ be an element of $R$. We claim the rows $C_{x\cdot}$ and $C_{(x\vee z)\cdot}$ are equal when restricted to the positions in $I$. To show that this is true, it suffices to check that $C_{x\cdot}\cdot [P(y)]=C_{(x\vee z)\cdot} \cdot [P(y)]$ for all $y\in I$. And this is true since $C_{x\cdot}\cdot [P(y)]$ is $-1$ if $y\geq x$ and 0 otherwise, and, for $y\in I$, this is the same as $-1$ if $y\geq x\vee z$ and 0 otherwise. 
This shows that every row $C_{x\cdot}$, restricted to the columns of $I$, agrees with the restriction of a row indexed by an element of $I$, namely $C_{(x\vee z)\cdot}$.
Hence there are at most
\(|I|\) possible restrictions of rows of \(C\) to the columns in \(I\). Since
\(C\) has a Bruhat decomposition with \(U_1=\id\), for each column \(i\in I\)
there is a row of \(C\) whose first non-zero entry is in column \(i\). Restricting
these rows to the columns in \(I\), we obtain \(|I|\) distinct restricted rows,
whose first non-zero entries occur in the different columns of \(I\). Thus the
possible restrictions to columns in \(I\) form an upper-triangular
\(|I|\times |I|\) matrix, where the order of the columns is inherited from
\(R\). Suitably ordering the rows gives the Coxeter matrix of \(I\) itself.
Thus, \(I\) must also have a Bruhat decomposition with \(U_1=\id\).
\end{proof}

\begin{corollary}\label{red-cor} To show that if $R$ is not distributive then its Bruhat decomposition cannot have $U_1=\id$, it is sufficient to show the statement for a non-distributive lattice all of whose proper upper intervals are distributive. \end{corollary}

\begin{proof} If $R$ is not distributive, we can choose $I$ minimal among the upper intervals of $R$ which is not distributive. Viewed as a lattice, $I$ has all its proper upper intervals distributive. So, if we can show that, for such a lattice, the Bruhat decomposition $U_1PU_2$ does not have $U_1=\id$, then by Lemma \ref{reduction}, $R$ does not have $U_1=\id$ either. \end{proof}

Therefore, to establish the result of this section, we only need to verify that if $R$ is distributive, then $U_1=\id$, and if $R$ is non-distributive but has all its proper upper intervals distributive, then $U_1\ne \id$. Both of these statements involve lattices all of whose proper upper intervals are distributive. It will be useful to give a structure theorem for such lattices, and then to give an expression for their Coxeter matrices.

\begin{lemma} \label{comb-of-lattices} Let $R$ be a lattice, all of whose proper upper intervals are distributive. Let $M$ be the meet-irreducible elements of $R$. There is a collection $\mathcal S$ of order filters of $M$, and a bijection $\zeta$ from $R\setminus \{\hat 0\}$ to $\mathcal S$, such that \begin{itemize}
    \item If $A$ and $B$ are order filters and $A\subseteq B\in \mathcal S$, then $A\in \mathcal S$.
    \item $\zeta$ is an order isomorphism from $R\setminus \{\hat 0\}$ to $\mathcal S$ ordered by reverse inclusion.
\item The join of two elements of $R\setminus \{\hat 0\}$ corresponds to the intersection of the corresponding order filters.
\item The meet of two elements of $R\setminus \{\hat 0\}$ corresponds to union of the corresponding order filters, unless the union is not contained in $\mathcal S$, in which case the meet is $\hat 0$.
\item If $M$ has a minimum element $m$, then $m$ belongs to no order filter in $\mathcal S$. For any other element $x\in M$, we have that $\langle x\rangle \in \mathcal S$.

\end{itemize}
\end{lemma}

\begin{proof} This is immediate from the fundamental theorem of finite distributive lattices. 
The map $\zeta$ sends $x\in R\setminus\{\hat 0\}$ to the set of meet-irreducibles above $x$. \end{proof}

\begin{example}\label{nondistlatt} An example of a lattice of the form given in Lemma \ref{comb-of-lattices}, and which is non-distributive, is
the following:

$$
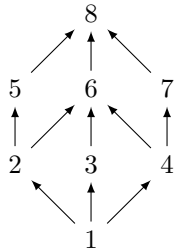
\begin{tikzpicture}[>=latex,line join=bevel,->]
  \node (a) at (0,0) {8};
  \node (b) at (-1,-1) {5};
  \node (c) at (0,-1) {6};
  \node (d) at (1,-1) {7};
  \node (e) at (-1,-2) {2};
  \node (f) at (0,-2) {3};
  \node (g) at (1,-2) {4};
  \node (h) at (0,-3) {1};
  \draw (h) -- (g); \draw (h) -- (f); \draw (h) --(e); \draw (g)--(c);\draw (f)--(c); \draw (e)--(c); \draw (e)--(b); \draw (g)--(d); \draw (b)--(a); \draw (c)--(a);\draw (d)--(a); \end{tikzpicture}$$

A representation of the type guaranteed by Lemma \ref{comb-of-lattices} is the following:

$$
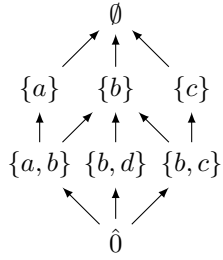
\begin{tikzpicture}[>=latex,line join=bevel,->]
  \node (a) at (0,0) {$\emptyset$};
  \node (b) at (-1,-1) {$\{5\}$};
  \node (c) at (0,-1) {$\{6\}$};
  \node (d) at (1,-1) {$\{7\}$};
  \node (e) at (-1,-2) {$\{5,6\}$};
  \node (f) at (0,-2) {$\{3,6\}$};
  \node (g) at (1,-2) {$\{6,7\}$};
  \node (h) at (0,-3) {$\hat 0$};
  \draw (h) -- (g); \draw (h) -- (f); \draw (h) --(e); \draw (g)--(c);\draw (f)--(c); \draw (e)--(c); \draw (e)--(b); \draw (g)--(d); \draw (b)--(a); \draw (c)--(a);\draw (d)--(a); \end{tikzpicture}$$

Here $M=\{3,5,6,7\}$ and the poset structure on $M$ is that induced from $R$

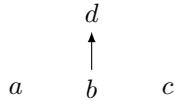
$$\begin{tikzpicture}[>=latex,line join=bevel,->]
  \node (a) at (-1,1) {$5$};
  \node (b) at (0,0) {$3$};
  \node (c) at (1,1) {$7$};
  \node (d) at (0,1) {$6$};
  \draw (b) -- (d); \end{tikzpicture}$$ \end{example}

The following lemma allows us to conclude that a lattice all of whose proper upper intervals are distributive is in fact distributive if it satisfies one particular combinatorial condition which is otherwise inconvenient to deal with.

\begin{lemma}\label{exclude}
    Let $R$ be a lattice all of whose proper upper intervals are distributive, and suppose that $R$ has a meet-irreducible element $m$ smaller than all other meet-irreducibles. Then $R$ is distributive.
\end{lemma}

\begin{proof} Since all the other meet-irreducibles are greater than $m$, their meet must be greater than or equal to $m$, but since $m$ is meet-irreducible, it must be strictly greater than $m$. Thus $m$ is the minimum element of $R$. It is covered by a single element, the meet of all the other meet-irreducibles. The lattice $R$ consists therefore of a proper upper interval of $R$ with a new minimum element adjoined. Since the proper upper interval is distributive, so is $R$. \end{proof}

Let $R$ be a lattice with all proper upper intervals distributive. From now on, when working with such a lattice, we always apply Lemma \ref{comb-of-lattices}, and then identify $R$ with $\mathcal S \cup \{\hat 0\}$.

For $x_1,\dots,x_n$ elements of $M$, let $\langle x_1,\dots,x_r\rangle$ be the order filter of $M$ generated by $x_1,\dots,x_r$.
For $Y$ an order filter in $\mathcal S$, we write $\max Y^c$ for the maximal elements of the complement of $Y$ in $M$. We write $\Delta(Y)$ for the simplicial complex on $\max Y^c$ consisting of all subsets $F$ such that $\langle F\rangle \in\mathcal S$. 

For $\Delta$ a simplicial complex, write $\chi^\red(\Delta)$ for the reduced Euler characteristic of $\Delta$. By definition, this is a sum over all faces $F$ of $\Delta$ of $(-1)^{|F|}$. 

The following lemma describes the Coxeter matrix of a lattice with all proper upper intervals distributive.

\begin{lemma} \label{cox} Let $R$ be a lattice with all proper upper intervals distributive, described as in Lemma \ref{comb-of-lattices}.
\begin{enumerate}
\item Let $Y\in\mathcal S$. The row of the Coxeter matrix of $R$ corresponding to $Y$ has entries:
\begin{itemize} \item For $F$ in $\Delta(Y)$, $(-1)^{|F|+1}$ in column $\langle F\rangle,$ 
\item $\chi^\red(\Delta(Y))$ in column $\hat 0$,
\item 0 otherwise.
\end{itemize}
\item The row corresponding to $\hat 0$ is zero except for a $-1$ entry in column $\emptyset$. 
\end{enumerate}
\end{lemma}

\begin{proof}
In the following, we check that a matrix with the proposed entries, call it $C$, indeed satisfies the defining property of the Coxeter matrix, namely, that $C\cdot \underline{\dim}P(i)=-\underline{\dim}I(i)$ for all $i\in R.$
We go row-by-row, showing that, for each $j \in R$, we have that $C_{j\cdot}\cdot \underline{\dim}P(i)=-\underline{\dim}I(i)_j$ for all $i\in R.$ 
Here, $C_{j\cdot}$ stands for the row of $C$ corresponding to $j.$

(1) We start with the case when $j=Y$ for some $Y\in \mathcal{S}.$ Let $i=Z$ for some $Z\in \mathcal{S}.$  The support of the dimension vector corresponding to a projective indecomposable for a $Z\subseteq Y$ and the support of the proposed row only intersect at $\emptyset$, so the dot product of such a projective indecomposable with the proposed row is $-1$. For a projective indecomposable corresponding to $Z\not\subseteq Y$, the intersection of the support of the dimension vector with the proposed row consists of 
$\langle F\rangle$, for all subsets $F$ of the non-empty set $Z\cap \max Y^c$. The dot product is therefore zero, as desired. 
It remains to consider the case $i=\hat 0.$
For the projective corresponding to $\hat 0$, the dot product is the sum of all the entries of the proposed row, which is 0 by the definition of the reduced Euler characteristic. The claim is proved.

(2) The case when $j=\hat{0}$ is an easy calculation.
\end{proof}

\begin{example}
  For the lattice of Example \ref{nondistlatt}, the Coxeter matrix is

  $$\left[\begin{array} {cccccccc}
    0&0&0&0&0&0&0&-1\\
    -1&0&1&0&0&0&1&-1\\
    -1&0&0&0&1&0&1&-1\\
    -1&0&1&0&1&0&0&-1\\
    0&0&0&-1&0&1&1&-1\\
    -2&0&1&0&1&0&1&-1\\
    0&-1&0&0&1&1&0&-1\\
    0&-1&0&-1&1&1&1&-1\end{array}\right]$$
It can easily be confirmed that this example agrees with the previous lemma. \end{example}

For the remainder of this section, we assume that the rows and columns of the Coxeter matrix  of $R$ are ordered with respect to a fixed linear extension of $R$. That means that if $i<j$ in $R$ then the column corresponding to $i$ appears to the left of the column corresponding to $j$ (and the row corresponding to $i$ appears above the row corresponding to $j$).

Now let $R$ be a lattice all of whose proper upper intervals are distributive, but which need not be distributive. We continue to use Lemma \ref{comb-of-lattices} to identify $R$ with $\mathcal S \cup \{\hat 0\}$.
We will now define a map from $R$ to $R$, as follows.
For $X$ an order filter in $\mathcal S$, define $\wor(X)=\langle \max (M\setminus X)\rangle$, provided $\langle\max(M\setminus X)\rangle \in \mathcal S$. Otherwise, define $\wor(X)=\hat 0$. Define $\wor(\hat 0)=\emptyset$. 

\begin{lemma} 
If $R$ is a distributive lattice, then $\wor=\row^{-1}$. 
\end{lemma}

\begin{proof} If $R$ is a distributive lattice, then there is an isomorphism $\mu$ from $R$ to the lattice of order filters of $M$, its poset of its meet-irreducibles, ordered by reverse inclusion. Its join-irreducibles are therefore the complements of the principal order ideals in $M$. Thus, its poset of join-irreducibles is isomorphic to $M$, and there is an isomorphism $\nu$ from $R$ to the lattice of order ideals of $M$, ordered  by inclusion. The composition $\mu^{-1}\circ \nu$ from order ideals to order filters is, naturally, complementation. Now it is clear that $\row$ (defined in terms of order ideals) and $\wor$ (defined in terms of order filters) are inverse. \end{proof}

If $R$ is not a distributive lattice, then nothing guarantees that $\wor$ defines a permutation of $R$. For any $X\in\mathcal S$, there is at most one order filter $Y$ with $\wor(Y)=X$, but this $Y$ might not be in $\mathcal S$. To compensate, there could be multiple different sets $Y\in\mathcal S$ such that $\wor(Y)=\hat 0$.

\begin{theorem} \label{distributiveImpliesL=id}
    Let $R$ be a distributive lattice and $C$ its Coxeter matrix. Then $C=PU$ for a permutation matrix $P$ and an upper triangular matrix $U$, and the Coxeter permutation $P$ corresponds to the rowmotion bijection on $R.$
\end{theorem}
\begin{proof}
Since $R$ is distributive, all its upper intervals are distributive, and 
we continue to identify $R$ with its description from Lemma \ref{comb-of-lattices}. Because $R$ is distributive, $\mathcal S$ consists of all order filters of $M$ other than $M$ itself. 

    Let $Y\in \mathcal{S}.$ If $\langle\max Y^c\rangle \in \mathcal{S}$, the faces of $\Delta(Y)$ are all the subsets of $\max Y^c$. By Lemma \ref{cox}, the entry in column $\hat{0}$ is zero, and the leftmost non-zero entry in row $Y$ is in column $\langle \max Y^c\rangle =\wor(Y)$.

    We must also consider the case when $\langle \max Y^c\rangle \notin \mathcal{S}$ for some $Y\in \mathcal{S}.$ The only occasion when this happens is when $\langle\max Y^c\rangle=M.$ That happens exactly when $\max Y^c$ consists of all the minimal elements of $M$.
    Then the faces of $\Delta(Y)$ are all subsets of $\max Y^c$ except $\max Y^c$ itself. Thus, by Lemma \ref{cox}, the entry in column $\hat{0}$ of row $Y$ is non-zero. We again have $\hat{0}=\wor(Y)$.  
    
    Finally, we obtain from the same lemma that the leftmost non-zero entry in row $\hat{0}$ is in column $\emptyset$. 
    
    Thus, we have shown that, for each $Y\in \mathcal S\cup\{\hat 0\}$, the leftmost non-zero entry in row $Y$ is in column $\wor(Y)$. Since $\wor$ is a permutation, this shows that $C$ can be written as  $PU_2$ with $P$ a permutation matrix and $U_2$ upper triangular, with the 1's of $P$ in row $Y$, column $\wor(Y)$, so $P$ encodes the permutation $\wor^{-1}=\row$.
\end{proof}



\begin{proposition} \label{rows} Let $R$ be a lattice with all proper upper intervals distributive, described as in Lemma \ref{comb-of-lattices} and suppose that $U_1=\id$ in a Bruhat decomposition $U_1PU_2$ of the Coxeter matrix of $R$. 
For $Y \in \mathcal S \cup\{\hat 0\}$, we show that, in the Coxeter matrix of $R$, the leftmost non-zero entry of row $Y$ is in column $\wor(Y)$.
  \end{proposition}

This proposition will be used in the proof of Theorem \ref{L=idImpliesDistributive} to show that the only lattices satisfying its hypotheses are in fact the distributive lattices.

\begin{proof} 
The proof proceeds by induction on $\wor(Y)$: we assume that it is shown for all $Z\in \mathcal S$ with $\wor(Z)$ after $\wor(Y)$ in the linear extension, and prove that it also holds for $Y$.

We therefore begin with the case that $\wor(Y)=\emptyset$. Since $M\notin\mathcal S$, the only element of $R$ with $\wor(Y)=\emptyset$ is $\hat 0$.
By Lemma \ref{cox}, the $\hat 0$ row has indeed its leftmost non-zero entry in column $\wor(\hat 0)=\emptyset$.


If $X=\langle x\rangle$ is meet-irreducible in $R$, let $Y$ be the row whose leftmost non-zero entry is in column $X$. Such a $Y$ exists by the assumption that $U_1=\id.$  By Lemma \ref{cox}, this means that $x\in \max Y^c$. We divide into cases based on the number of elements in $\max Y^c$. 

If $\max Y^c=\{x\}$, then $X=\wor(Y)$, as desired. 
(It might seem that this disposes of the case that $X=\langle x\rangle$, but that is not true. There is certainly a subset $Y$ of $M$ with $\max Y^c=\{x\}$, namely the complement of the order ideal generated by $x$, but this subset might not be in $\mathcal S$, so it might not correspond to a row of the Coxeter matrix. To rule this out, we must show that there is no other row whose leftmost non-zero entry could be in column $X$.)

Suppose next that $\max Y^c=\{x,y\}$. If $\langle x,y\rangle$ belonged to $\mathcal S$, then the row corresponding to $Y$ would have an entry there by Lemma \ref{cox}, which would be to the left of the entry in column $\langle x\rangle$, contradicting our assumption on $Y$. This means that $\Delta(Y)$ consists just of two points, $x$ and $y$. Its reduced Euler characteristic is therefore $-1$, so it has a non-zero entry in the $\hat 0$ column, again contradicting our assumption on $Y$, so the case that $\max Y^c$ has exactly two elements is ruled out.

Suppose now that $\max Y^c$ has at least three elements, and suppose three of them are $x,y,z$. Since $X$ is the leftmost non-zero entry of the row $Y,$ $y$ and $z$ must come after $x$ with respect to the chosen linear order on $\mathcal S$. Let $Z$ be the row whose leftmost nonzero entry is in position $\langle z\rangle$. By induction, $\wor(Z)=\langle z\rangle$, so $Z$ is the complement of the order ideal generated by $z$. 
Since $x,y,z$ form an antichain, both $x$ and $y$ belong to $Z$. Thus $\langle x,y\rangle \subseteq Z\in\mathcal S$, so $\langle x,y\rangle \in \mathcal S$. But the column $\langle x,y\rangle$ is to the left of the column $\langle x\rangle$, and it would have a non-zero entry in the row $Y$ by Lemma \ref{cox}. This case can therefore also be ruled out, and $\max Y^c$ must consist of the single element $\{x\}$.

Now suppose that $X$ is not meet-irreducible.  Suppose the minimal elements of $X$ are $x_1,\dots,x_r$, with $r>1$. By induction, if $W$ is the row with leftmost nonzero entry in column $\langle x_1\rangle$, then by induction $\wor(W)=\langle x_1\rangle$, and $W$ is the complement of the order ideal generated by $x_1$. 
There is a unique order filter $Y$ with $\wor(Y)=X$, and $Y\subseteq W$, so $Y\in \mathcal S$ as well.
We know $\max Y^c=\{x_1,\dots,x_r\}$. Since $\langle x_1,\dots,x_r\rangle=X$ belongs to $\mathcal S$ by assumption,  $\Delta(Y)$ is the $(r-1)$-dimensional simplex on $x_1,\dots,x_r$. Since the reduced Euler characteristic of a simplex is 0, by Lemma \ref{cox}, the leftmost entry of the row corresponding to $Y$ is in the column corresponding to $X$, as desired.

Finally, we consider the case that $\wor(Y)=\hat 0$. Since for all other $Y'$ with $\wor(Y')\in \mathcal S$, we have already seen that the leftmost entry in row $Y'$ is in column $\wor(Y')\in\mathcal S$, we know that the leftmost entry in row $Y$ must be in the one remaining column, which is not in $\mathcal S$, i.e., in column $\hat 0$, as desired.
\end{proof}

\begin{corollary} If $R$ is a lattice with all proper upper intervals distributive, and with $U_1=\id$ in the Bruhat decomposition $U_1PU_2$ of the Coxeter matrix of $R$, then $\wor$ is a bijection.\end{corollary}

\begin{proof} This follows from the previous lemma because $\wor(Y)$ is the position of the leftmost non-zero entry in row $Y$, and the leftmost non-zero entries are all in different rows by the assumption that $U_1=\id$.
\end{proof}

\begin{theorem} \label{L=idImpliesDistributive} Let $R$ be a lattice whose Coxeter matrix has a Bruhat decomposition $U_1PU_2$ with $U_1=\id$. Then $R$ is distributive. \end{theorem}

\begin{proof} 
  By Corollary \ref{red-cor}, it suffices to assume that $R$ is of the form described by Lemma \ref{comb-of-lattices}. We continue to identify $R$ with $\mathcal S\cup \{\hat 0\}$. By Lemma \ref{exclude}, if $R$ has a minimum meet-irreducible element, then it is distributive. 
  
  Assume otherwise. Let the minimal elements of $M$ be $x_1,\dots,x_r$ with $r\geq 2$. Since none of these elements is the minimum element of $M$, we know that $\langle x_i\rangle\in \mathcal S$ for all $i$.

Applying Proposition \ref{rows}, the row whose leftmost non-zero entry is in column $\langle x_i\rangle$ is $\wor^{-1}(\langle x_i\rangle)=M\setminus\{x_i\}$. Thus, for all $i$, we have that $M\setminus\{x_i\}\in \mathcal S$. Thus $\mathcal S$ consists of all order filters of $M$ except for $M$ itself. It follows that $R$ is isomorphic to the order filters of $M$ (ordered by reverse inclusion), and is in fact distributive.\end{proof}

\section{Two open questions}
In this final section, we pose two questions that are supported by some computer experiments.
\begin{question} 
 Let $A$ be a 2-Gorenstein acyclic quiver algebra of finite global dimension with an admissible ordering of the simple modules and assume $C$ has a Bruhat decomposition $U_1PU_2$ with $U_1=\id$. Is $A$ then Auslander regular?
\end{question}
We remark that this does not hold when we drop the assumption that $A$ is 2-Gorenstein
as the following example shows:

\begin{example}

Let $Q$ be the following quiver and set $A=kQ/\rad^2(kQ).$
\[\begin{tikzcd}
	& 2 \\
	4 & 1 \\
	& 3
	\arrow[from=1-2, to=2-1]
	\arrow[from=2-2, to=1-2]
	\arrow[from=2-2, to=2-1]
	\arrow[from=2-2, to=3-2]
	\arrow[from=3-2, to=2-1]
\end{tikzcd}\]
This labeling of the vertices of the quiver gives an admissible ordering on the simples ($S(1)<S(2)<S(3)<S(4)$) as $\grade(S(2))=\grade(S(3))=\grade(S(4))=0<\grade(S(1))=1.$ 
The Coxeter matrix $C$ is given by 
\[ \left( \begin{array}{cccc}
0 & 0 & 0 & -1\\
0 & 0 & 1 & -1\\
0 & 1 & 0 & -1\\
-1 & 1 & 1 & -1 \end{array} \right) \]
and a Bruhat decomposition $U_1PU_2$ is given by:
$C=PU_2$ with $P=$
\[ \left( \begin{array}{cccc}
0 & 0 & 0 & 1\\
0 & 0 & 1 & 0\\
0 & 1 & 0 & 0\\
1 & 0 & 0 & 0 \end{array} \right) \]
and $U_1$ the identity matrix and $U_2=$
\[ \left( \begin{array}{cccc}
-1 & 1 & 1 & -1\\
0 & 1 & 0 & -1\\
0 & 0 & 1 & -1\\
0 & 0 & 0 & -1 \end{array} \right). \]
But $A$ is not Auslander regular, as it is not even $1$-Gorenstein.
This follows from the fact that all injective $A$-modules have projective dimension $2.$
\end{example}

The next question asks whether we can charactise Auslander regular incidence algebras of posets using the Coxeter matrix as we did for lattices. We remark that a classification of Auslander regular posets is not known in general and some partial results are work in progress. Note that being 1-Gorenstein for an incidence algebra $KR$ is equivalent to the poset $R$ being bounded by \cite[Proposition 5.1]{IM}.
\begin{question}
Is the incidence algebra of a bounded poset with linear extension ordering of the simple modules Auslander regular if and only if $U_1=\id$ in a Bruhat decomposition $U_1PU_2$ of the Coxeter matrix?
\end{question}

\section*{acknowledgement}
This project profited from the use of \cite{Sage} and the GAP-package \cite{QPA}. We thank the anonymous referee for useful suggestions and comments.

\end{document}